\documentclass[11pt, reqno]{amsart}
\usepackage{amsmath, amsthm, amssymb, amscd, mathrsfs, amsfonts, bbm, 
graphicx, tikz, tikz-cd, dsfont, pgf, hyperref, mathtools, scalerel, float, 
subfigure, multicol, csquotes, xcolor, subfloat}
\usepackage[margin=3cm]{geometry}
\usepackage[utf8]{inputenc}
\usepackage[T1]{fontenc}
\usepackage[font=small,labelfont=bf]{caption}
\makeatletter \renewcommand{\fnum@figure}{Fig. \thefigure} \makeatother
\newtheorem{theorem}{Theorem}[section]

\newtheorem{examples}[theorem]{Examples}

\newtheorem{lemma}[theorem]{Lemma}

\newtheorem{proposition}[theorem]{Proposition}
\newtheorem{remark}[theorem]{Remark}

\newtheorem{Fact}[theorem]{Fact}
\numberwithin{equation}{section}
\numberwithin{figure}{section}

\begin{document}
\title[Quasi-vertex-transitive Maps on the Plane ]{Quasi-vertex-transitive Maps on the Plane}
\author{Arun Maiti}
\address{Arun Maiti, Department of Mathematics, Indian Institute of Science, 
Bangalore 560 012}
\email{arunmaiti@iisc.ac.in}
\keywords{Vertex-transitive maps, Polyhedral maps, Aperiodic tiles, Semi-
regular tilings}
\subjclass[2010]{52C20; 52B70; 51M20; 05B45 }

\begin{abstract}
Quasi-vertex-transitive maps are the homogeneous maps on the plane
with finitely many vertex orbits under the action of their automorphism
groups. We show that there exist quasi-vertex-transitive maps of
types $[p^3, 3]$ for $p \equiv 1$ (mod $6$), but 
there doesn't exist vertex-transitive map of such types. In particular, we 
determine the surface with the lowest possible genus that admit a polyhedral 
map of type $[5^3, 3]$. 
\end{abstract}

\maketitle

\section{Introduction} 
The subject of maps on surfaces lies at the interfaces of discrete geometry, 
graph theory and combinatorial topology. Maps with symmetry (e.g., vertex-transtive maps)
are objects of interest to group theorists. The concept of quasi-vertex-transitive
map is a natural generalization of the concept of vertex-transitive map on the plane. 
 \newline 
A \textit{map} on a surface $S$ ($2$-manifold) is a simple, connected, locally 
finite graph $G$ embedded on $S$ satisfying: 1) the closure of each 
connected component of $S\setminus G$, called a \textit{face}, is 
homeomorphic to a closed $2$-disc, 2) each vertex of $G$ has degree at least 
$3$. 
\newline 
A vertex $v$ of a map is said be \textit{polyhedral} if the intersection of any
two distinct faces incident to $v$ is either $v$ or an edge incident to $v$.
A map is said to be polyhedral if all its vertices are polyhedral \cite{BS97}. 
\newline
The \textit{vertex-type} of a vertex $v$ of a map is a cyclic sequence, denoted 
by a cyclic-tuple $[k_1, k_2,$ $ \cdots, k_d]$, of the lengths of the faces 
around $v$. A vertex-type and its mirror image will be considered to be the 
same. A map is said to be \textit{homogeneous} if all its vertices are of the 
same type \cite{GS}. The \textit{type} of a homogeneous map is the 
vertex-type of its vertices. A vertex-type is usually written in multiplicative form, e.g., $ 
[5, 5, 3, 3]$ $=[5^2, 3^2]$. Homogeneous maps on the plane are always 
polyhedral (see Lemma \ref{homopol}). 
 \newline
 An \textit{isomorphism} between two maps on a surface $S$ is an isomorphism between
 their underlying graphs that extends to a self-homeomorphism of $S$.
 We say that a map $K$ has \textit{$m$ vertex orbits} if the action of the group
 of automorphisms of $K$, $Aut(K)$, on the set of vertices of $K$ has $m$ orbits.
 \textit{Vertex-transitive maps} are the maps with single vertex orbit. They are obviously 
homogeneous. A map is assumed to be homogeneous, unless otherwise stated.
\newline
A homogeneous map $K$ on the plane is said to be \textit{quasi-vertex-transitive} if it 
has finitely many vertex orbits, or equivalently, it is isomorphic to lift of a homogeneous map 
on some closed surface (see Proposition \ref{vertrancor}). 
We define the \textit{minimal characteristic}(respectively
\textit{minimal polyhedral characteristic})
of a quasi-vertex-transitive map $K$ to be the smallest number $-\chi$ such that $K$ is the
lift of a map (resp. polyhedral map) on a surface (which exists by Lemma \ref{liftmap}) with Euler characteristic $\chi$  .
\newline
It is known that there doesn't exist vertex-transitive maps on the plane
of types $[p^3, 3]$, $p \equiv 1$ (mod $6$) (\cite{GS79}, Proposition \ref{verp}).
Here, we apply the Poincaré polygon theorem to prove the following. 
\begin{theorem}\label{p3qth}There exist quasi-vertex-transitive maps on the plane of 
types $[ p^3, 3]$, $p$ odd, $p \geq 5$.
\end{theorem}
To the author's knowledge, $[p^3, 3]$ with $p \equiv 1$ (mod $6$) are
the first known examples of cyclic-tuples that are the types of quasi-vertex-transitive
maps but not the types of vertex-transitive maps on the plane.
\newline
The cyclic-tuple $[5^3, 3]$ is of special importance to us and we have the following. 
\begin{theorem}\label{53th}There exists a quasi-vertex-transitive map on the plane of 
type $[ 5^3, 3]$ with the minimal characteristic $1$ and the minimal polyhedral
characteristic $2$. 
\end{theorem}
 This implies, in particular, that all the hyperbolic surfaces ($\chi<0$) admit a 
map of type $[5^3, 3]$, but none of them are vertex-transitive, the type is unique in 
that sense. 
\newline
A map $K$ on a surface $S$ is said to be \textit{geometric} if $S$ admits a 
metric with constant curvature with respect to which the faces of $K$ are 
regular polygons in $S$. One of the interesting facts about homogeneous 
maps is that they can always be made geometric (see Fact \ref{geom}). So, 
the results of this paper can also be presented in a geometric way.
Tilings produced by geometric vertex-transitive and homogeneous maps on a surface are known as uniform and semi-regular tilings respectively.
\newline
It is not very difficult to construct maps of types $[p^3, 3]$, $p \geq 5$ on the plane
with infinitely many vertex orbits. Our work was motivated by the problem of finding a
finite aperiodic set of regular polygons (tiles) for the hyperbolic plane,
or even stronger, a cyclic-tuple $\mathfrak{k}$ such that there exist homogenous
maps on the plane of the type $\mathfrak{k}$ but none of them are quasi-vertex-transitive, thereby strengthening the result of \cite{MM98}. 
\section{Preliminaries}\label{preli}
 A cyclic-tuple $\mathfrak{k} = [k_1, k_2, \cdots, k_d]$ with $d \geq 3$ is said to be of the 
\textit{spherical, Euclidean, or hyperbolic type} if the \enquote{angle sum} $ 
\alpha(\mathfrak{k})=\sum^{d}_{i=1}\frac{k_{i} -2}{k_{i}} $ is $<2, =2$, or $>2$ 
respectively. 
 \begin{Fact}[\cite{LB97, DG18}] \label{fact1} For a cyclic-tuple $\mathfrak{k} = 
[k_1, k_2, \cdots, k_d]$ with $d \geq 3$ and $\alpha(\mathfrak{k})<2,=2,$ or $ >2$, there exist 
spherical, Euclidean, or hyperbolic regular $k_i$-gons (with geodesic sides)
for $i=1, \cdots, d$, respectively, which fits around a vertex, or equivalently,
their inner angles $\theta_i$ sum to $2 \pi$, that is, $ \sum^d_{i=1} \theta_i=2 \pi$. 
 \end{Fact}
A map on a surface lifts to a map on one of the universal covers, the sphere, $S^2$
or the plane. On the other hand, any geometric map on a surface is quotient of 
a geometric map on one of the universal covers, $S^2$, the Euclidean plane, 
$\mathbb{E}^2$, or the hyperbolic plane, $\mathbb{H}^2$, under the free action of 
a discrete group of isometries. Thus, Fact \ref{fact1} implies the following. 
\begin{Fact} [\cite{LB97, CG08, DG18}] (Geometrization) \label{geom}
A homogeneous map on a surface $S$ is isomorphic to a geometric map on 
$S$. Moreover, each automorphism of a geometric map on $S$ corresponds 
to an unique isometry of $S$.
\end{Fact}
If a cyclic-tuple $\mathfrak{k}$ is the type of a homogeneous map on the 
sphere, then by the Euler formula $\alpha(\mathfrak{k})<2$. All such types are 
known, they correspond to the boundaries of Platonic \& Archimedean solids, 
the prism, the pseudorhombicuboctahedron, and additionally two infinite 
families, $ [4^2, r]$ for some $r \geq 5$ and $[3^3, s]$ for some $s \geq 4$
(antiprisms) \cite{LB91, BG09}. 
\newline
If a cyclic-tuple $\mathfrak{k}$ is the type of a homogeneous map on the 
plane, then it can be deduced from the Euler formula for finite planar graph 
that $\alpha(\mathfrak{k}) \geq 2$ \cite{CG08}. If $\alpha(\mathfrak{k}) =2$, 
then $K$ is isomorphic to a geometric map on the Euclidean plane, there
are exactly $11$ such types \cite{GS77}. The maps on the torus and the Klein 
bottles are naturally quotients of maps of those $11$ types \cite{DM18}.
\newline
The following fact should be well-known though the author was not able to 
trace a reference. 
\begin{lemma}\label{homopol} 
Homogeneous maps on simply connected surfaces are always polyhedral. 
\end{lemma}
\begin{proof}Let $K$ be a homogeneous map on a simply connected
surface of type $\mathfrak{k}=[k_1, k_2,$ $ \cdots, k_d]$. By Fact \ref{geom} above,
we may assume that $K$ is a geometric map on $\mathbb{E}^2$, $\mathbb{H}^2$ or
 $S^2$, depending on $\alpha(\mathfrak{k}) =2, >2$ or $<2$ respectively.
 A geodesic ray on a surface $S$ with initial point $x \in S$ is a geodesic
 $ \gamma: [0, \infty) \rightarrow S$ with $\gamma(0)=x$.
\newline
Suppose $K$ is a geometric map on $\mathbb{E}^2$ or $\mathbb{H}^2$.
Then for any vertex $v$ of $K$, the geodesic rays with initial point $v$ that
extend the $d$ edges incident to $v$, divide $\mathbb{E}^2$ or $\mathbb{H}^2$
respectively into $d$ connected components. So, the interior of each of the faces
(regular and therefore convex) surrounding $v$ lies in only one of these components.
Note that by definition of a map, $d \geq 3$. It follows that $K$ is polyhedral.
\newline
Suppose $K$ is a geometric map on $S^2$. Then for any vertex $v$ of $K$, the
shortest geodesics (or the semi-circles) between $v$ and its antipodal point
$v_{\infty}$ that extend the $d$ edges incident to $v$, divide $S^2$ into $d$
connected components. The interior of each of the faces surrounding $v$ lies
in only one of these components. Since two distinct great circles meet only at
antipodal points and the underlying graph of $K$ is simple, therefore none of the faces
(regular) surrounding $v$ have $v_{\infty}$ as a vertex. It follows that $K$ is polyhedral.
 \end{proof}
But the same cannot be said about homogeneous maps on non-simply 
connected surfaces, as we shall see an example in Figure \ref{53}(a) below. 
The following folklore, however, asserts that polyhedrality is a weak condition 
for homogeneous maps on any surface. 
\begin{lemma}\label{liftmap}Let $K$ be a homogeneous map on
a surface $S$. Then, there is a finite covering $\tilde{S}$ of $ S$
such that the lift of $K$ to $\tilde{S}$ is a polyhedral map. 
\end{lemma}
\begin{proof} 
The proof essentially follows from the above lemma and the residual
finiteness property of the surface groups. If $S$ is the real projective
plane, the sphere, or the plane, then the proof directly follows
from the above lemma. In all other cases, i.e., if $S$ is closed
and $\chi(S) \leq 0$, then a map $K$ on $S$ lifts to a map $K_U$ on the universal cover
of $S$, the plane. Let $F$ be a minimal subset of the set of faces
of $K_U$ that generates $K_U$ under the action of the fundamental group $\pi_1(S)$. 
Here one may assume that the union of the faces in $F$ is a
fundamental domain for $\pi_1(S)$.
Let $C$ be the union of the faces that intersect any face in $F$. Since $\pi_1(S)$ is
residually finite \cite{PS78}, there is a finite covering $f: \tilde{S} \rightarrow S$
such that the covering map $\tau_{\tilde{S}}: \mathbb{R}^2 \rightarrow \tilde{S}$ is
injective on the compact set $C$. This means that all the vertices in $\tau_{\tilde{S}}(F)$
are polyhedral. Now, since every finite index subgroup of a group $\mathcal{G}$ contains
a finite index normal subgroup of $\mathcal{G}$, so we can choose $\tilde{S}$ to be
a normal covering. Then all the vertices in $\tau_{\tilde{S}} (K_U)$ are polyhedral. 
But $\tau_{\tilde{S}} (K_U)$ is nothing but the lift of $K$ to $\tilde{S}$, thus the proof 
follows. 
\end{proof}
\begin{remark} The proof of the above lemma also works for homogeneous 
maps with singular faces, faces that are not contractible but their interiors are. 
\end{remark}
Now, as mentioned in the introduction, we prove the following. 
\begin{proposition}\label{vertrancor}
A homogeneous map on the plane is quasi-vertex-transitive if and only if it is 
isomorphic to lift of a map on some closed surface. 
\end{proposition}
\begin{proof} The "only if" part of the statement is obvious. For the "if" part, let $K$
be a homogeneous map on the plane of the hyperbolic type (resp. the Euclidean 
type) with $m$ vertex orbits. By Fact \ref{geom}, we may assume that $K$ is geometric 
and $Aut(K)$ to be a subgroup of isometries of $\mathbb{H}^2$ (resp. $\mathbb{E}^2$).
Due to the fact that $Aut(K)$ is also the automorphism group of the dual of $K$ 
acting with $m$ face orbits, the fundamental polygon, $F$ of the dual map consists 
of $m$ faces of the dual. This also means that $F$ is fundamental polygon for 
a subgroup $\mathcal{H}_1$ of $Aut(K)$. $\mathcal{H}_1$ must be 
finitely generated because fundamental polygon of infinitely generated 
subgroup of isometries of $\mathbb{H}^2$ (resp. $\mathbb{E}^2$) have 
infinitely many sides. 
\newline
Now by Selberg's lemma \cite{AL87}, $\mathcal{H}_1$ has a finite index, 
torsion free subgroup $\mathcal{H}_2$ of $Aut(K)$. This means $\mathcal{H}
_2$ is a discrete subgroup of isometries of $\mathbb{H}^2$ (resp. $\mathbb{E}
^2$) acting freely on $\mathbb{H}^2$ (resp. $\mathbb{E}^2$). It follows that 
$K/ \mathcal{H}_2$ is a homogeneous map on the closed surface $\mathbb{H}
^2/ \mathcal{H}_2$ (resp. $\mathbb{E}^2/ \mathcal{H}_2$).
\end{proof}
The work of enumeration of vertex-transitive maps on closed surfaces was 
done in \cite{LB91} in terms of their genus. The Euler formula produces a 
bound on the number of possible types of homogeneous maps on a surface. 
For a fixed type, the number of maps on a surface with $\chi(S) \neq 0$ is 
finite. Since closed surfaces are countable by their genus and cross-cap 
numbers, so the number of quasi-vertex-transitive maps on the
plane are countable (up to isomorphism).
\begin{examples}
The examples of homogeneous maps that appear in the literatures are mostly
vertex-transitive or have infinitely many vertex orbits. In \cite{EEK1, EEK2},
the authors showed that the maps of types $[p^q]$ and $[2p_1, 2p_2, \cdots, 2p_k]$
are vertex-transitive, and also determined their minimal characteristics.
One can construct vertex-transitive maps of various other types by applying standard
operations (e.g. truncation, rectification, cantellation) on these types. Since most of
these operations are reversible, one can also determine the minimal characteristics of
the derived maps.
\end{examples}
\begin{examples}
Quasi-vertex-transitive maps of a fixed type are, in general, not unique. In fact, there 
are non-isomorphic vertex-transitive maps of the same type, e.g., $[(6n)^3, 
3]$, $[4^3, 6]$.
\end{examples}
\section{Quasi-vertex-transitive maps of types $[p^3, 3]$} \label{periodic}
Among all the homogeneous maps on the sphere, only the map of type $[4^3, 3]$
is not vertex-transitive, and it corresponds to the boundary of pseudorhombicuboctahedron
\cite{BG09}. All the homogeneous maps on the plane of the Euclidean types are
vertex-transitive, but most of the maps of the hyperbolic types are not so. Here we give
a simple proof of the following result from \cite{GS79}. 
\begin{proposition} \label{verp}
There doesn't exist vertex-transitive map on the plane of type $[p^3, 3]$,
for $p \equiv 1$ mod($6$).
\end{proposition}
\begin{proof} 
We say that a $p$-gon in a map $K$ of type $[p^3, 3]$ is of type-$k$ if the number of
triangles intersecting the $p$-gon is $k$. Suppose there are more than one type of 
$p$-gons in $K$. Since $p$ is odd by our hypothesis, each $p$-gon has at least one vertex
such that a triangle intersect the $p$-gon only by that vertex. Consider two such vertices of two different types
of $p$-gons in $K$, then clearly no automorphism of $K$ can map one to the other, so $K$
is not vertex-transitive. Suppose now that all $p$-gons in $K$ are of the same type, say of the
type-$l$. Let us consider a disc region $V_R$ in the plane of radius $R$ containing a large number
of vertices of $K$. Then we have a count $ct(e)$ of edges of $K$ that are common to a $p$-gon
and a triangle in the region $V_R$. We also have a count $ct(v)$ of vertices that are the only
intersection of a $p$-gon and a triangle in $V_R$. Since each triangle contribute $3$ to
both the counts, therefore $ct(e) \sim ct(v)$ for sufficiently large $R$. If each $p$-gon contributes
$i$ and $j$ (with $i+j=l$) to $ct(e)$ and $ct(v)$ respectively, then we must have $i=j$. It is easy to
see that $i \neq j$ unless $p$ is multiple of $3$. Since $p$ is odd, therefore the proposition follows. 
\end{proof} 
We shall now prove Theorem \ref{p3qth} by applying the Poincaré polygon theorem. The theorem,
in particular, gives sufficient conditions for a convex hyperbolic polygon with finitely many sides
(geodesic segments) to be a fundamental polygon for a discrete subgroup of isometries of the
hyperbolic plane. We refer the readers to \cite{BE95} for more details about the theorem.
\begin{proof}[\textbf{Proof of Theorem \ref{p3qth}}] We will construct a 
fundamental polygon for the dual of a geometric map of type $[p^3, 3]$ for 
$p \geq 5$, $p$ odd. By Fact \ref{fact1}, for the cyclic-tuple $[p^3, 3]$,
 we have a configuration of three regular
hyperbolic $p$-gons and one regular hyperbolic triangle around a vertex,
 see Figure \ref{fdpol1}. Let $P_4:=\{v_1, v_2, v_3, v\}$ be the 
 hyperbolic quadrilateral obtained by joining the incenters of the faces by 
 geodesic segments, where $v_2$ is the incenter of the triangle. Then we have
 $\angle v_1= \angle v_3= \angle v =2 \pi/p$, $\angle v_2 = 2 \pi/3$. So we can form a cyclic sequence of
 \begin{figure}[H] 
 \begin{minipage}{.50\textwidth} \hspace{1cm}
\includegraphics[scale=1.2]{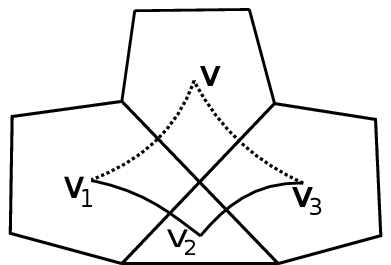}%
\captionof{figure}{Quadrilateral dual face $P_4$ for the type $[5^3, 3]$}
\label{fdpol1} \vspace{-2.8cm}
\end{minipage}%
\begin{minipage}{.50\textwidth} \hspace{0.5cm}
\includegraphics[scale=1.0]{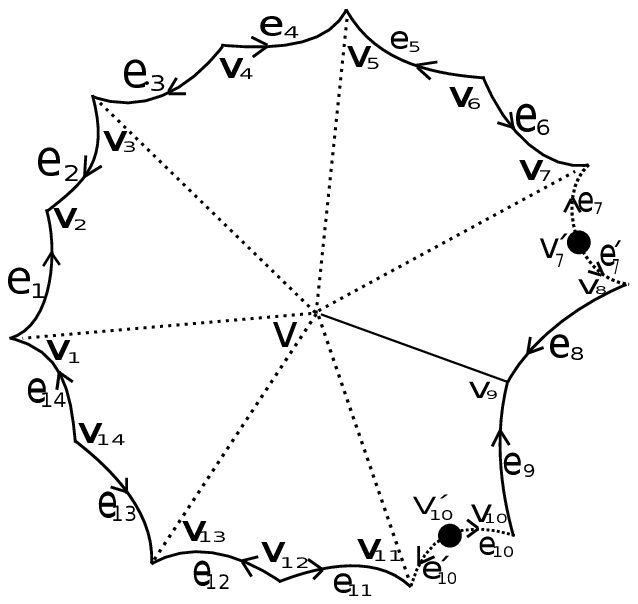}%
\captionof{figure}{Fundamental polygon $F(7)$ for the type $[7^3, 3]$ }
\label{fdpol2}
\end{minipage}
\end{figure}
polygons $[V_1, V_2, \cdots, V_{p}]$ around the vertex $v$ such that
a) $V_{(p-1)/2}= P_4$, b) $V_i$ is obtained by reflecting $P_4$ successively in its
dotted sides $\{v, v_1\}$ and $\{v, v_3\}$, for 
$i=1, 2, \cdots, \frac{p-1}{2}-1$, $\frac{p-1}{2}+ 1, \cdots, p-2$, c) $V_{p-1}$ and
$V_{p}$ are obtained by rotating $V_{p-2}$ and $V_1$ respectively about the
midpoint of their dotted sides by an angle $\pi$. Then the edges of the faces around
$v$ that are not adjacent to $v$ form a $2p$-gon which we may denote
by $\{v_1, v_2,$ $ \cdots, v_{2p}\}$, see Figure \ref{fdpol2} for $p=7$.
We shall now prove that the hyperbolic $(2p+2)$-gon $F(p)$ obtained from $2p$-gon
$\{v_1, v_2,$ $ \cdots, v_{2p}\}$ by declaring the
midpoints of its two dotted sides $\{v_p, v_{p+1}\}$ and $\{v_{p+3}, v_{p+4}\}$
as vertices, denoted by $v'_p$ and $v'_{p+3}$ respectively,
is our required fundamental polygon. 
Let $e_i$ denote the side $\{v_i, v_{i+1}\}$ for $i\neq p, p+3$, and 
$e_p$, $e'_{p}$, $e_{p+3}$, $e'_{p+3}$ denote the sides $\{v_p, v'_p \}$, $\{v'_p,
v_{p+1}\}$, $\{v_{p+3}, v'_{p+3}\}$, $\{v'_{p+3}, v_{p+4} \}$
respectively. The sides are equipped with orientations given by the arrows pointing
outwards from the vertices $ v'_{p}$, $v'_{p+3}$ and $ v_i$, for $i=4, 6, \cdots, p-1, p+5, 
p+7, \cdots, 2p$, and inward to the vertices $v_2$ and $v_{p+2}$. Then we 
consider the side-pairing isometries $[e_1, e_{p+2}]$, $ [e_2, e_{p+1}]$,
$[e_{p}, e'_{p} ]$, $[e_{p+3}, e'_{p+3}]$, 
$ [e_{i}, e_{i-1} ]$ for $ i=4, 6, \cdots $ $ p-1, $ $p+5, p+7, \cdots, 2p$. So
the elliptic cycles are
\[ \{ v_2, v_{p+2}\}, \{ v_3, v_{5}, \cdots, v_{p}, v_{p+1} \}, \{ v_{p+3}, v_{p+4}, 
v_{p+6}, \cdots v_{2p-1}, v_1\}, \{ v'_{p}\}, \{ v'_{p+3}\},\] \[ \{v_i\} \ 
\mathrm{for} \ i=4, 6, \cdots, p-1, p+5, p+7, \cdots, 2p. \]
We verify that
\begin{multline} \angle v_2 + \angle v_{p+2} = 2 \pi, \hspace{0.5cm} \angle 
v_3+ \angle v_{5}+ \cdots + \angle v_{p} + \angle v_{p+1}= 2 \pi, \\ \angle 
v_{p+3}+ \angle v_{p+4}+ \angle v_{p+6}+ \cdots +\angle v_{2p-1} + \angle v_1 
= 2 \pi, \\ \angle v'_{p}= \angle v'_{p+3}=\pi, \hspace{0.5cm} \angle v_i = 2 
\pi/p \hspace{0.3cm} \mathrm{for} \ i=4, 6, \cdots, p-1, p+5, p+7, \cdots, 2p.
 \end{multline}

So, all the elliptic cycles are proper elliptic cycles. Therefore, by the Poincaré 
polygon theorem, $F(p)$ is the fundamental polygon for the group 
of isometries of the hyperbolic plane generated by the side pairings. Thus we 
have a geometric map $K$ on $\mathbb{H}^2$ whose faces are copies 
of $F(p)$. It is easy to see that the map obtained by subdividing 
each face of $K$ into quadrilaterals $P_4$s is the dual of a map $\tilde{K}$ 
of type $[p^3, 3]$. Clearly, the map $\tilde{K}$ has at most $p$ vertex orbits. 
This completes the proof of the theorem.
 \end{proof}
 
\section{ Polyhedral maps of type $[5^3, 3]$} 
The vertex-transitive polyhedral maps on orientable surfaces of genus $2$, 
$3$ and $4$ have been studied in \cite{KN12}. The question of the existence 
of polyhedral maps of all the possible types on the surface with $\chi=-1$ has 
been settled except for the type $[5^3, 3]$ (in a private communication with 
Dipendu Maity, will appear elsewhere). Here we take up the task of verifying 
the remaining case. 

\begin{proof}[\textbf{Proof of Theorem \ref{53th}}] Let us assume that $K$ be a 
polyhedral map of type $[5^3, 3]$ on the surface with $\chi=-1$. Then the 
number of vertices, edges and faces are $-15, -30$ and $-14$ respectively. 
The number of triangles and pentagons in $K$ are $5$ and $9$ respectively, 
and the triangles are mutually disjoint. We call a pentagon is of type-$n$ if the 
number of triangles adjacent to it is $n$. Let $p_n$ be the number of type-$n$ 
pentagons in $K$ for $n=3, 4, 5$. Then we have 
\begin{equation}\label{typeq} 
p_3 + p_4 + p_5=9. \end{equation} 
Let us count the pairs $(P, T)$ such that $P$ is a pentagon adjacent to the 
triangle $T$. Since the total number of pentagons adjacent to each triangle is 
$6$, so the number of distinct pairs $(P, T)$ counted in two ways yield the 
equation
 \begin{equation}\label{typect} 3p_3+4p_4+5p_5=6 \times 5=30. 
\end{equation}
 Solving the equations \ref{typeq} and \ref{typect}, we obtain $(p_3, p_4, 
p_5)=(6, 3, 0)$ or $(7, 1, 1).$ 
\vspace{-0.3cm} \begin{multicols}{2} \noindent
\textbf{\underline{Case 1:}} 
We consider the possibility that there are three mutually disjoint type-$4$ 
pentagons $P_1, P_2, P_3$ in $K$. Then, the union of their vertices comprises
all $15$ vertices of $K$. Consequently, every vertex has one and only one
type-$4$ pentagon adjacent to it. Starting with one such
type-$4$ pentagon, say $P_1$, a part of $K$ can be represented by Figure 
\ref{disj}, with $P_2$ and $P_3$ being other two type-$4$ pentagons. 
\begin{figure}[H]
 \begin{center}
 \includegraphics[scale=0.40]{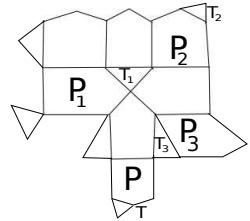}
 \end{center}
 \caption{Part of $K$ of Case 1}
 \label{disj}
\end{figure}
\end{multicols}\vspace{-0.5cm}
It follows that the triangles $ T_1, T_2, T_3$ are mutually disjoint. 
Since $K$ is a polyhedral map therefore the type-$4$ pentagon $P \neq P_1, P_3$.
But $P = P_2 \implies T$=$T_2 \implies$ $T_1=T_3$, a contradiction, therefore $P \neq 
P_2$ either. This is again impossible since $p_4 = 3$. 
\newline
\textbf{\underline{Case 2:}} If $K$ does not have three mutually disjoint type-$4$ pentagons,
then for both the solutions of $(p_3, p_4, p_5)=(6, 3, 0)$
or $(7, 1, 1)$ above, there is a vertex, say $v_1 \in V(K)$ which is not a vertex of any of the
type-$4$ pentagons. In such a situation, without loss of generality, a portion of $K$ can 
be represented by Figure \ref{nondj3}. The possible values of $\mathrm{a}, 
\mathrm{b}=v_{2}, v_{4}, v_{5}, v_{14},v_{15}$, and $\mathrm{c}, 
\mathrm{d}=v_{3}, v_{10}, v_{11}, v_{12}, v_{13}$. We now check their viability. 
 \newline
 \textbf{1)} $\mathrm{c}=v_{3}$ implies either $\mathrm{v}=2$ or $\mathrm{d}
=2$, which is not possible.
 \newline
 \textbf{2)} Similarly, $\mathrm{b} = v_2$ is not possible by the symmetry of 
Figure \ref{nondj3} about the vertex $v_1$. 
\newline
\textbf{3)} $\mathrm{c}=v_{11}, \mathrm{d}=v_{3} \implies \mathrm{u}=v_{13}$ 
($\mathrm{u} \neq v_{7}, v_{8}$ for degree reason) $\implies \mathrm{b}
=v_{14} \implies \deg({v_{14}})>4$, which is not possible.
\newline
\textbf{4)} $\mathrm{c}=v_{12}, \mathrm{d}=v_{13} \implies$ no choice left for $
\mathrm{b}$ (as $v_{13}, v_{10} \in \{v_{9}, v_{10}, v_{b}, v_{a}, v_{13} \})$.
\newline
\textbf{5)} $\mathrm{c}=v_{13}, \mathrm{d}=v_{12} \implies$ no choice left for $
\mathrm{u}$.\newline
\textbf{1)}, \textbf{2)} and \textbf{3)} together implies that $(\mathrm{c}, 
\mathrm{d})$ must be an edge of the triangle $\{v_{10}, v_{11}, v_{12}\}$, and again 
by the symmetry of Figure \ref{nondj3}, $(\mathrm{a}, \mathrm{b})$ of the 
triangle $\{v_{4}, v_{5}, v_{15}\}$. 
\newline
\textbf{6)} $(\mathrm{c}, \mathrm{d}) \neq (v_{10}, v_{11}), (v_{11}, v_{10})$ as 
the pentagon
$\{v_{10}, v_{11}, v_{1}, v_{3}, v_{9}\}$ $\neq \{v_{10}, v_{11}, v_{5}, v_{6}, 
v_{14}\}$.\newline
\textbf{7)} $\{\mathrm{c}, \mathrm{d}\}= \{v_{11}, v_{12}\} \implies \mathrm{u}
=v_{13} \implies \mathrm{a}, \mathrm{b}=v_{15}, v_{5} \implies \deg(v_{5})>4$. 
\newline
\textbf{8)} $\{\mathrm{c}, \mathrm{d}\}= \{v_{12}, v_{11}\} \implies \mathrm{b}
=v_{4} \implies$ no choice left for $\mathrm{u}$. \newline
\textbf{9)} $\{\mathrm{c}, \mathrm{d}\}= \{v_{12}, v_{10}\} \implies$ $
(\mathrm{a}, \mathrm{b})$ is not an edge of the triangle $\{v_{4}, v_{5}, 
v_{15}\}$. \newline
\textbf{10)} $\{\mathrm{c}, \mathrm{d}\}= \{v_{10}, v_{12}\}$ leads to Figure 
\ref{longcycle5}. 
 \begin{figure}[H]
 \begin{minipage}{.40\textwidth}
\includegraphics[scale=0.30]{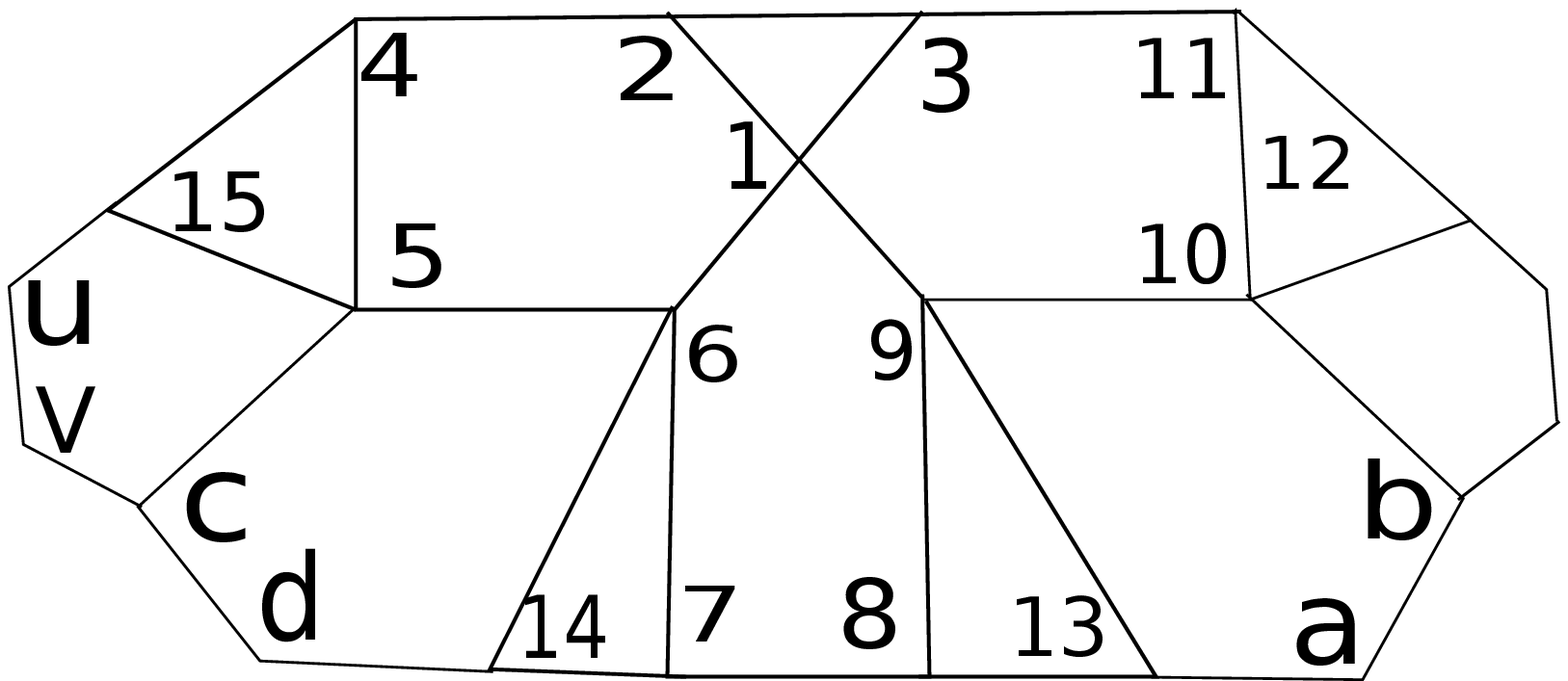}
\captionof{figure}{Part of $K$ of Case 2}
\label{nondj3}
\end{minipage}%
\begin{minipage}{.60\textwidth}
\includegraphics[scale=0.50]{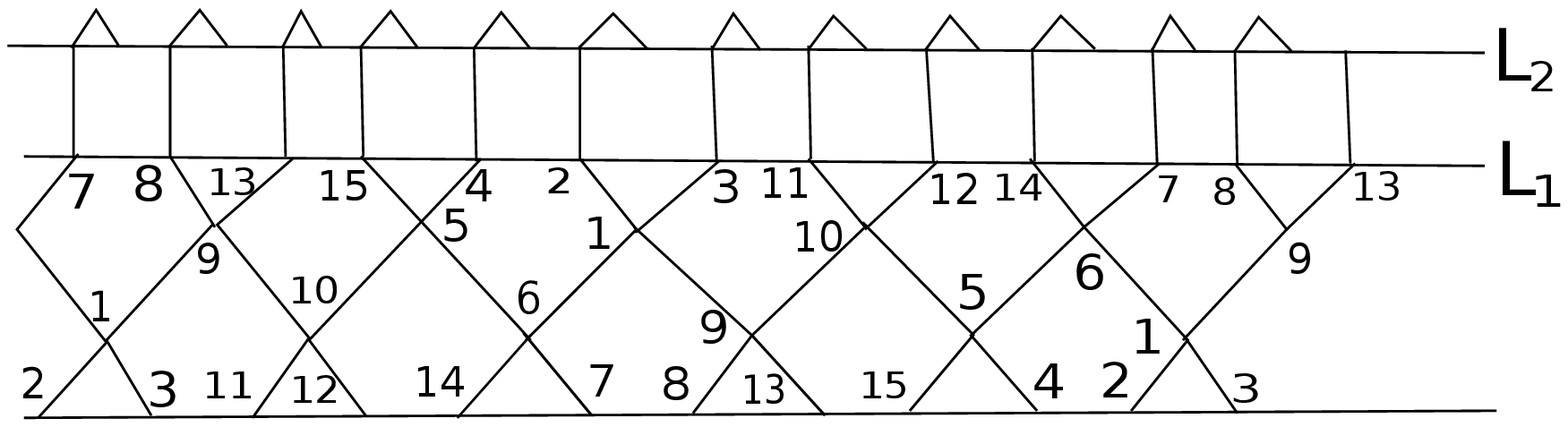}
\captionof{figure}{Cyclic sequence for $\{\mathrm{c}, \mathrm{d}\}= \{v_{10}, 
v_{12}\}$ }
\label{longcycle5}
\end{minipage}
\end{figure}
We observe now that, because of cyclic repetition of vertices on 
the horizontal line $L_1$, the triangles on the horizontal line $L_2$ must be 
placed either in the shown order or all translated by one edge. In both the cases, if 
we form a graph with vertices representing the triangles of the map and edges 
representing type-3 pentagons having common 
edges with two triangles, then we would have a graph with $5$ vertices each 
having degree $3$ and $6$ edges, which is absurd. 
\newline
Thus, we have exhausted all the possibilities in Case 2 as well. This leaves us 
with no choice but concluding that $K$ is not polyhedral. Without 
the polyhedral condition, however, the example in Figure \ref{53}(a) is a map of type 
$[5^3, 3]$ on the surface with $\chi=-1$, $ \mathbb{R} P^2 \# 
\mathbb{R} P^2 \#\mathbb{R} P^2$. This completes the proof of the first 
part of the theorem. \newline
 In view of Lemma \ref{liftmap}, the example in Figure \ref{53}(a) suggests that 
for a polyhedral map we could look on surfaces with $\chi=-2$, i.e., on double 
torus $\mathbb{T}^2$ (orientable) and $\mathbb{R} P^2 \# \mathbb{R} P^2 \# 
\mathbb{R} P^2 \#\mathbb{R} P^2$ (non-orientable). Indeed, we found the 
following examples shown in Figures \ref{53}(b) and \ref{53}(c).

\begin{figure}[H]
\centering
\subfigure[Non-polyhedral]{
 \includegraphics[scale=0.35]{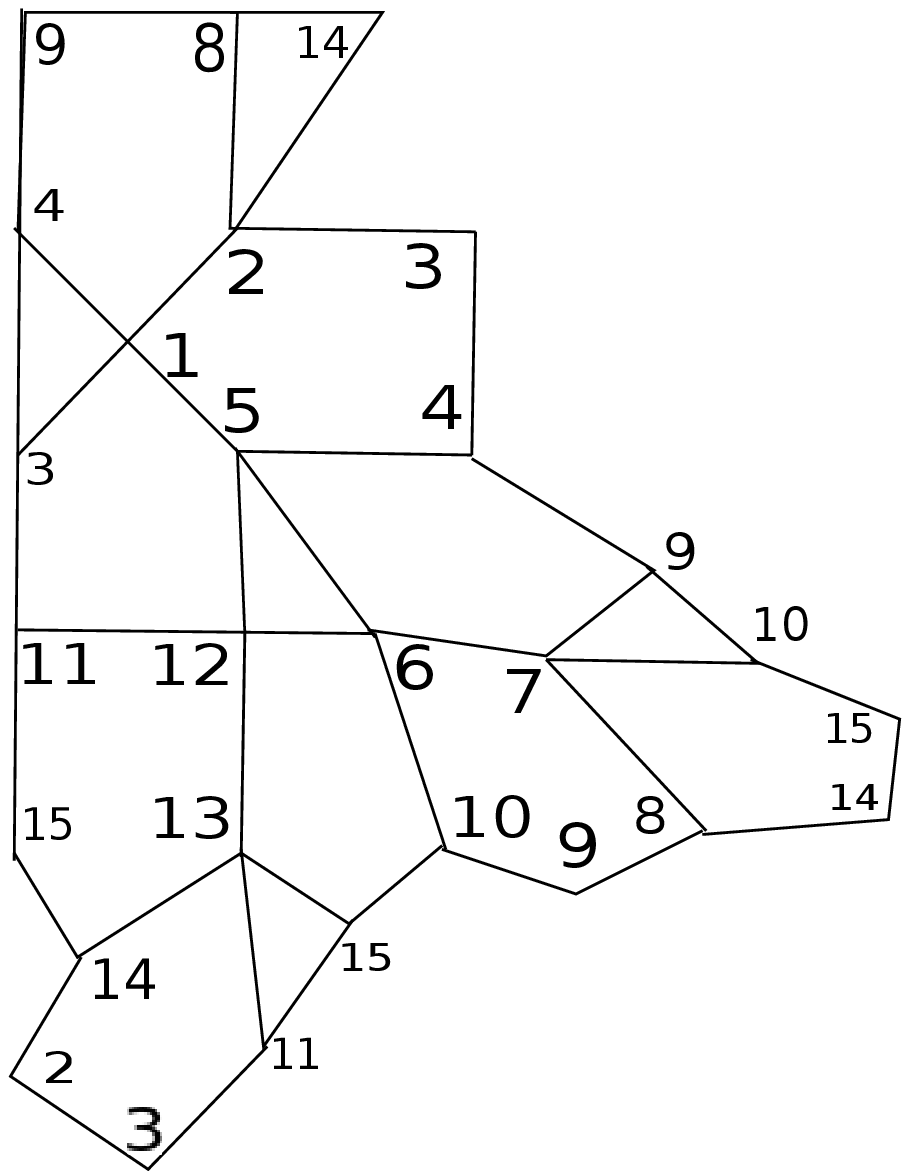}%
} \hspace{1cm} 
\subfigure[Orientable]{
 \includegraphics[scale=0.35]{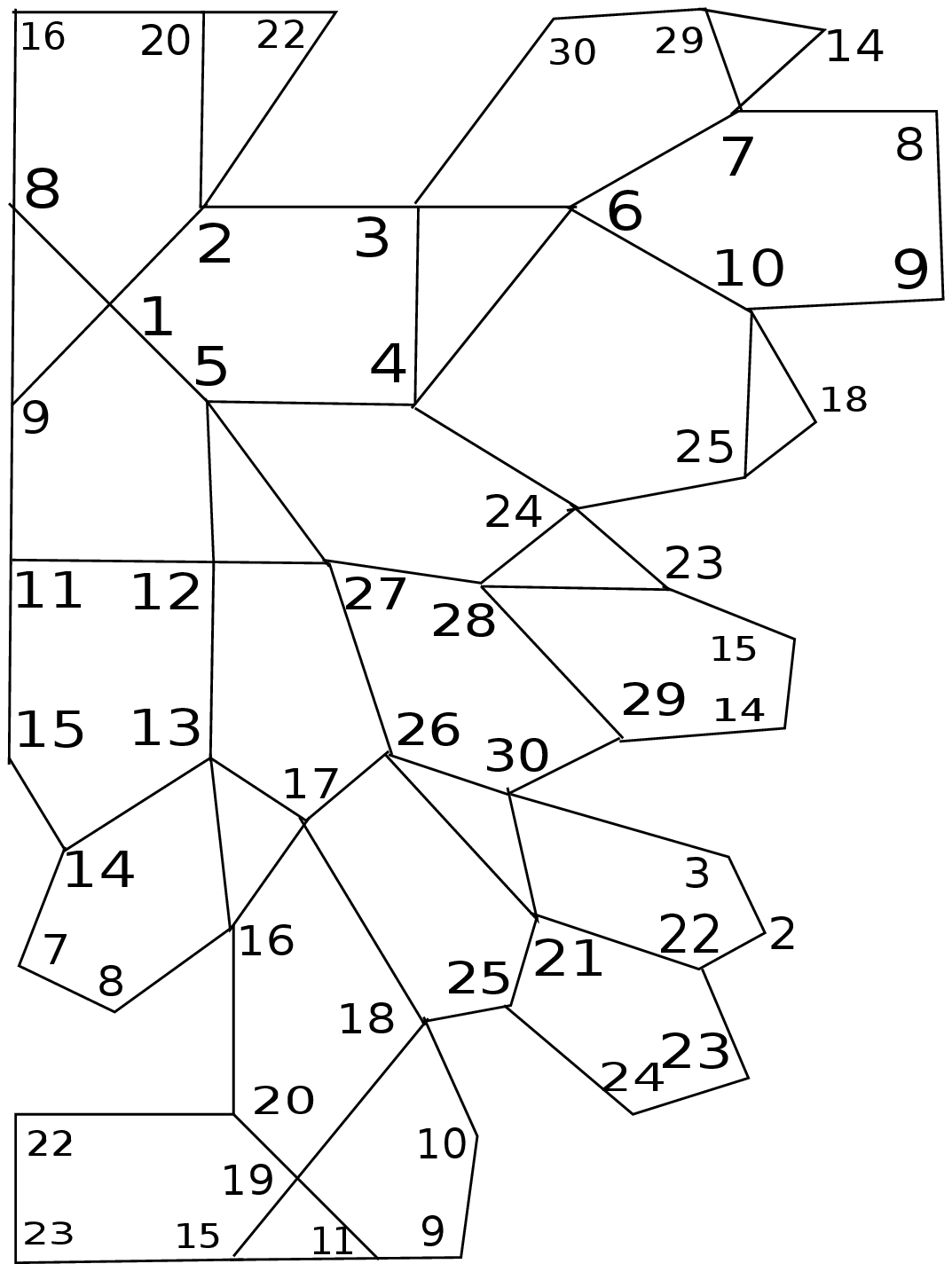}%
} \hspace{1cm}
\subfigure[Non-orientable]{
 \includegraphics[scale=0.35]{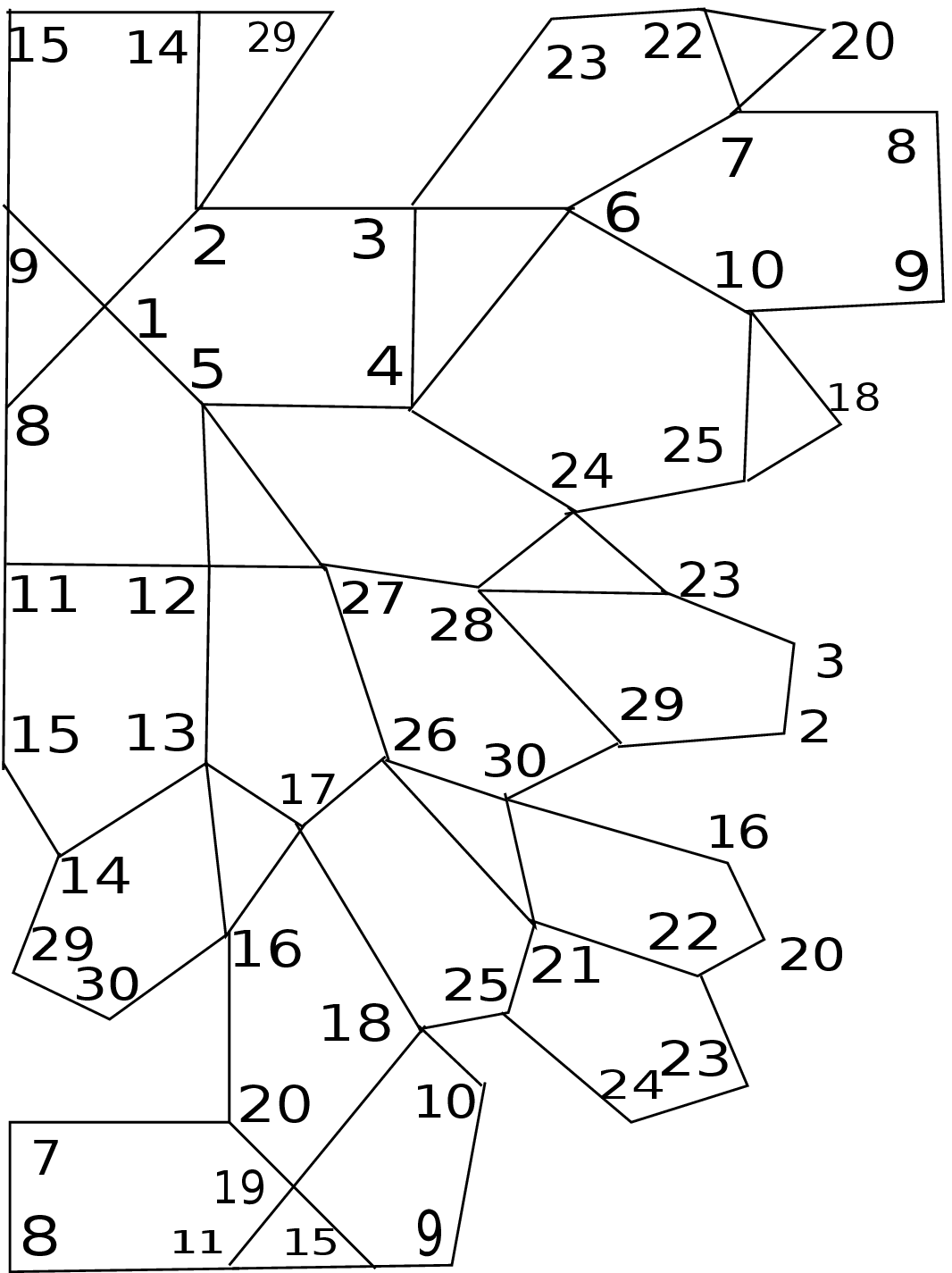}%
} \label{nonori5}
\caption{ Maps of type $[5^3, 3]$ on surface with $\chi=-1$ and $\chi=-2$} 
\label{53} 
\end{figure}

This completes the proof of Theorem \ref{53th}.
\end{proof}

\section{Acknowledgements} 
The author is grateful to Basudeb Datta and Subhojoy Gupta, for many 
valuable suggestions. The author wishes to thank an anonymous reviewer
for a number of helpful comments and suggestions. The author was supported
by the National Board of Higher Mathematics (No. 4485/2017/NBHM) during
the academic year 2018--2019.

\bibliographystyle{amsplain}

\bibliography{bibtile}
\end{document}